\documentclass{amsart}

\usepackage{amsmath, amssymb, amsthm, graphicx, url}

\usepackage[noend,ruled,algo2e,linesnumbered]{algorithm2e}

\usepackage{hyperref,cleveref}
\usepackage{standalone}
\usepackage{tikz}
\usepackage{tikz-3dplot}
\title{Upper Bounds on Covering Minima of Convex Bodies}
\usepackage{times}

\author{Katarina Krivoku\'ca}
\address{Katarina Krivoku\'ca\\Freie Universität Berlin}
\email{katarina.krivokuca@fu-berlin.de}

\newtheorem{theorem}{Theorem}[section]
\newtheorem{proposition}[theorem]{Proposition}
\newtheorem{lemma}[theorem]{Lemma}
\newtheorem{corollary}[theorem]{Corollary}

\newtheorem{conjecture}[theorem]{Conjecture}

\theoremstyle{definition}
\newtheorem{definition}[theorem]{Definition}
\newtheorem{example}[theorem]{Example}
\newtheorem{notation}[theorem]{Notation}

\theoremstyle{remark}
\newtheorem{remark}[theorem]{Remark}

\usepackage{bbold}
\usepackage{enumerate}
\usepackage{enumitem}
\usepackage{mathtools}
\usepackage{tikz}

\usepackage{todonotes}
\newcommand{\A}{\mathcal{A}}
\newcommand{\R}{\mathbb{R}}

\newcommand{\N}{\mathbb{N}}

\newcommand{\Z}{\mathbb{Z}}

\begin{document}

\begin{abstract}	We give two new upper bounds on the covering minima of convex bodies, depending on covering minima of certain projections and intersections with linear subspaces. We show one bound to be sharp for direct sums of two convex bodies, generalizing previous results on the covering radius and lattice width of direct sums. We apply our results to standard terminal simplices, reducing the gap between the upper and lower bounds in a conjecture of Gonzal\' ez-Merino and Schymura (2017), which gives insight on a conjecture of Codenotti, Santos and Schymura (2021) on the maximal covering radius of a non-hollow lattice polytope.

\end{abstract}

\maketitle

\section{Introduction}
\subsection{Background}

For $1\leqslant i \leqslant d$, the \textit{i-th covering minimum} of a convex body $K\subseteq \R^d$ with respect to a lattice $\Lambda\subseteq\R^d$ is defined as 
$$\mu_i(K,\Lambda)=\min \left\{\mu\geq 0 \; : \; \left(\mu K +\Lambda \right)\cap U\neq \emptyset, \text{ for all }U\in \A_{d-i}(\R^d) \right\},$$
where $\A_{d-i}(\R^d)$ is the set of all $(d-i)$-dimensional affine subspaces of $\R^d$. 

\noindent A \textit{convex body} for us is a full dimensional, compact convex subset of $\R^d$. A \textit{lattice} is a full dimensional discrete subgroup of $\R^d$. If the lattice is omitted, it is assumed to be $\Z^d$. A convex body $K$ is \textit{hollow} with respect to a lattice $\Lambda$ if $\operatorname{int}(K)\cap \Lambda=\emptyset$.\\

The covering minima are classical parameters in the Geometry of Numbers, introduced by Kannan and Lov\' asz \cite{KannanLovasz} and used to prove the first polynomial bound on the \textit{flatness constant} introduced by Khinchine \cite{Khinchine}. The flatness constant is the maximum possible \textit{lattice width} of a hollow convex body of a fixed dimension. The width of a convex body $K$ with respect to the lattice $\Lambda$ is
$$\operatorname{width}_\Lambda(K)=\inf_{f\in \Lambda^*\setminus \{0\}}\;\max_{a, b\in K} f(a-b).$$

The study of the flatness constant was revived by Lenstra \cite{Lenstra}, who used it in a landmark result on solving Linear Integer Programming in fixed dimension in polynomial time. The exact value of the flatness constant is only known in dimension 2 (Hurkens \cite{Hurkens}) and a maximizer is conjectured in dimension 3 \cite{CodenottiSantos}, and proven to be a local maximizer in \cite{LocalMax}. \\

The study of hollow convex bodies is tightly related to the \textit{covering radius}, that is
$$\mu(K,\Lambda)=\min\left\{ \mu\geq 0\; :\; \mu K+\Lambda=\R^d \right\}.$$
The covering radius can be equivalently described as the greatest $\mu\geqslant 0$ such that the dilate $\mu K$ admits a hollow translate.

Calculating the covering radius of certain simplices was proven by Kannan in \cite{Kannan} to be equivalent to calculating the Frobenius numbers, and the covering radius of certain zonotopes relates to the Lonely Runners Conjecture as seen in \cite{MalikiosisSchymura}.\\

Kannan and Lov\' asz introduce the covering minima as a sequence of $d$ functionals on convex bodies which interpolate between the reciprocal of the lattice width and the covering radius. Indeed, they show 
$$\mu_1(K,\Lambda)=\frac{1}{\operatorname{width}_\Lambda(K)},$$
and for $i=d$, since intersecting all affine subspaces of codimension $d$ means covering all points, 
$$\mu_d(K,\Lambda)=\mu(K,\Lambda).$$

They also give an equivalent description of covering minima via covering radii of projections as follows:
 $$\mu_i(K, \Lambda)=\max\left\{ \mu(\pi(K), \pi(\Lambda ))\;| \; \pi \text{ is a rational linear projection of rank }i\right\}. $$

Notice that all covering minima are invariant under translations of the convex body $K$, and that for an invertible matrix $A\in GL_d(\R)$, we have $\mu_i(AK, A\Lambda)=\mu_i(K,\Lambda)$. Hence, the covering minima are invariant under \textit{unimodular transformations}, which are affine maps that map the lattice to itself, that is, they can be represented as $x\mapsto BUB^{-1}x+z$, where $z\in \Lambda$ and $U\in \operatorname{GL}_d(\Z)$ is a \textit{unimodular matrix}, and $\Lambda= B\Z^d$.

Additionally, notice that if a subset of $\R^d$ intersects all elements of $\A_{d-i}(\R^d)$ for some $1\leqslant i< d$, it also intersects all elements of $\A_{d-i+1}(\R^d)$. This implies that the covering minima are a monotone increasing sequence, that is
$$\mu_1(K,\Lambda)\leqslant \dots \leqslant \mu_d(K,\Lambda).$$

The polynomial upper bound on the flatness constant that Kannan and Lov\' asz give comes from applying inequalities that they prove for intermediate covering minima. However, the intermediate covering minima seem to be harder to grasp and have been studied more sparsely outside of the original work of Kannan and Lov\' asz than the lattice width and covering radius.\\

Gonz\' alez Merino and Schymura \cite{MerinoSchymura} reformulate a conjecture of Makai Jr. \cite{MakaiJr} and generalize it in the form of a lower bound on the \textit{covering product}, that is, the product of all covering minima of a given convex body and its volume. This can be seen as a dual form of the product of \textit{successive minima} and the volume found in Minkowski's second fundamental theorem. One of the conjectured minimizers of the covering product is the \textit{standard terminal simplex} $T_d=\operatorname{conv}(-\mathbb{1}_d, e_1,\dots, e_d)$, for which they believe $\mu_i(T_d)=\frac{i}{2}$. They prove the case $i=d$, that is, the covering radius of a standard terminal $d$ simplex is $\frac{d}{2}$.
 
Furthermore, Codenotti, Santos and Schymura \cite{CodenottiSantosSchymura} raise the question of what the maximal covering radius of a \textit{non-hollow lattice polytope} is. They conjecture it to be $\frac{d}{2}$ and prove this to be equivalent to $\mu_i(T_n)=\frac{i}{2}$, for all $i\leqslant d$ and $n\geqslant d$. They  prove that in dimensions up to $3$, the maximizers of the covering radius within the family of non-hollow lattice polytopes are unimodularly equivalent to \textit{direct sums} of translates of standard terminal simplices. We call these direct sums \textit{terminal polytopes}.

Let $\R^d=V\oplus W$ be a decomposition into complementary linear subspaces, and let $K\subseteq V$, $L\subseteq W$ be convex bodies both containing the origin. The \textit{direct sum} of $K$ and $L$ is defined as 
$$K\oplus L := \{ \lambda x +(1-\lambda)y\;:\; x\in K,\; y\in L\;, \lambda\in[0,1]\}\subseteq \R^d.$$
The direct sum of two lattices $\Lambda\subseteq V$ and $\Gamma \subseteq W$ is defined as:
$$\Lambda\oplus\Gamma:=\{a+b\;:\; a\in \Lambda,\;b\in\Gamma\}\subseteq \R^d.$$

Codenotti, Santos \& Schymura \cite{CodenottiSantosSchymura} additionally show that the covering radius is additive with respect to the direct sum. Work by Codenotti and Santos \cite{CodenottiSantos} shows that the lattice width of the direct sum is the minimum of the widths of the summands, which then translates to a maximum when talking about the first covering minimum. Motivated by this, we pose the question about what the $i$-th covering minimum of a direct sum with respect to covering minima of the summands is.

On the covering minima side of the forementioned conjectures, it is easy to see that for all $i<d$, $\mu_i(T_d)\geqslant \frac{i}{2}$, but proving that $\frac{i}{2}$ is an upper bound appears more involved. General upper bounds on covering minima have been given already in \cite{KannanLovasz}, but they depend on the successive minima of the difference body.  Note that in \cite{KannanLovasz} and \cite{AverkovWagner} the covering minima are shown to satisfy further inequality relations amongst each other. This is in contrast with the successive minima, which can take values of any increasing sequence of positive reals. This makes us suspect the existence of different inequalities for covering minima which do not depend on successive minima.

\subsection{Our Contributions}
Motivated by the context above, our goal is to investigate general upper bounds for covering minima of a convex body which depend only on covering minima of certain derived convex bodies.

\begin{figure}[h!]
    \centering

    \begin{minipage}{0.45\linewidth} 
        \centering
        \tdplotsetmaincoords{78}{22}

\begin{tikzpicture}[tdplot_main_coords, scale=1.2]

\draw[->, gray] (0,0,0) -- (3.2,0,0);
\draw[->, gray] (0,0,0) -- (0,3.2,0);
\draw[->, gray] (0,0,0) -- (0,0,3.2);

\draw[black]
(-1,-1,-1) -- (1,-1,-1) -- (1,1,-1) -- (-1,1,-1) -- cycle
(-1,-1, 1) -- (1,-1, 1) -- (1,1, 1) -- (-1,1, 1) -- cycle
(-1,-1,-1) -- (-1,-1, 1)
( 1,-1,-1) -- ( 1,-1, 1)
( 1, 1,-1) -- ( 1, 1, 1)
(-1, 1,-1) -- (-1, 1, 1);

\fill[gray!20, opacity=0.3]
(-3, 0, 3) --
( 0,-3, 3) --
( 3,-3, 0) --
( 3, 0,-3) --
( 0, 3,-3) --
(-3, 3, 0) -- cycle;

\node[gray!80, anchor=north west] at (3,-2.6,-1.5) {$V^\perp$};

\filldraw[blue!40, opacity=0.4]
( 1, 0,-1) --
( 1,-1, 0) --
( 0,-1, 1) --
(-1, 0, 1) --
(-1, 1, 0) --
( 0, 1,-1) -- cycle;

\node[blue!80, anchor=south west] at (0.8,-0.15,-0.85) {$C\cap V^\perp$};

\foreach \x in {-2,...,2}{
  \foreach \y in {-2,...,2}{
    \foreach \z in {-2,...,2}{
      \node[circle, fill=gray!60, inner sep=0.6pt] at (\x,\y,\z) {};
      \pgfmathparse{\x+\y+\z==0 ? 1 : 0}
      \ifnum\pgfmathresult=1
        \node[circle, fill=blue, inner sep=1.3pt] at (\x,\y,\z) {};
      \fi
    }
  }
}

\draw[thick] (-2.5,-2.5,-2.5) -- (2.5,2.5,2.5)
node[above right] {$V$};

\draw[red, very thick] (-1,-1,-1) -- (1,1,1);

\node[red!80, anchor=south east] at (0.6,1.2,0.55) {$\pi_{V}(C)$};

\foreach \x in {-2,...,2}{
  \foreach \y in {-2,...,2}{
    \foreach \z in {-2,...,2}{
      \pgfmathsetmacro{\t}{(\x+\y+\z)/3}
      \node[circle, fill=red, inner sep=1pt] at (\t,\t,\t) {};
    }
  }
}

\end{tikzpicture}\vspace{0.25cm}
        \shortstack{
            (a) Theorem \ref{thm: ub proj} for\\
            $[-1,1]^3$ and $V: x=y=z$.
        }
    \end{minipage}
    \hspace{0.05\linewidth} 
    \begin{minipage}{0.45\linewidth}
        \centering
        \tdplotsetmaincoords{78}{22}

\begin{tikzpicture}[tdplot_main_coords, scale=1.2]

\draw[->, gray] (0,0,0) -- (3.2,0,0);
\draw[->, gray] (0,0,0) -- (0,3.2,0);
\draw[->, gray] (0,0,0) -- (0,0,3.2);

\draw[black]
(-1,-1,-1) -- (1,-1,-1) -- (1,1,-1) -- (-1,1,-1) -- cycle
(-1,-1, 1) -- (1,-1, 1) -- (1,1, 1) -- (-1,1, 1) -- cycle
(-1,-1,-1) -- (-1,-1, 1)
( 1,-1,-1) -- ( 1,-1, 1)
( 1, 1,-1) -- ( 1, 1, 1)
(-1, 1,-1) -- (-1, 1, 1);

\fill[green!20, opacity=0.35] 
(-3,0,-3) -- (3,0,-3) -- (3,0,3) -- (-3,0,3) -- cycle;
\fill[green!40, opacity=0.7]
(-1,0,-1) -- (1,0,-1) -- (1,0,1) -- (-1,0,1) -- cycle;

\fill[blue!20, opacity=0.35] 
(-3,-3,0) -- (3,-3,0) -- (3,3,0) -- (-3,3,0) -- cycle;
\fill[blue!40, opacity=0.7]
(-1,-1,0) -- (1,-1,0) -- (1,1,0) -- (-1,1,0) -- cycle;

\fill[red!20, opacity=0.35] 
(0,-3,-3) -- (0,3,-3) -- (0,3,3) -- (0,-3,3) -- cycle;
\fill[red!40, opacity=0.7]
(0,-1,-1) -- (0,1,-1) -- (0,1,1) -- (0,-1,1) -- cycle;

\foreach \x in {-2,...,2}{
  \foreach \y in {-2,...,2}{
    \foreach \z in {-2,...,2}{
      
      \ifnum\x=0 \relax \else
        \ifnum\y=0 \relax \else
          \ifnum\z=0 \relax \else
            \node[circle, fill=gray!60, inner sep=0.6pt] at (\x,\y,\z) {};
          \fi
        \fi
      \fi
      
      \pgfmathsetmacro{\rout}{1.3pt}  
      \pgfmathsetmacro{\rmid}{0.8*\rout} 
      \pgfmathsetmacro{\rin}{0.6*\rout} 

      \ifnum\x=0
        \node[circle, draw=red, fill=red, inner sep=\rout] at (\x,\y,\z) {};
      \fi

      \ifnum\y=0
        \node[circle, draw=green, fill=green, inner sep=\rmid] at (\x,\y,\z) {};
      \fi

      \ifnum\z=0
        \node[circle, draw=blue, fill=blue, inner sep=\rin] at (\x,\y,\z) {};
      \fi

    }
  }
}

\end{tikzpicture}\vspace{0.25cm}
        \shortstack{
            (b) Theorem \ref{thm: ub intersec} for $[-1,1]^3$\\
		and $i$=2.
        }
    \end{minipage}

    \caption{Dependencies of our bounds}
    \label{fig: bounds}
\end{figure}
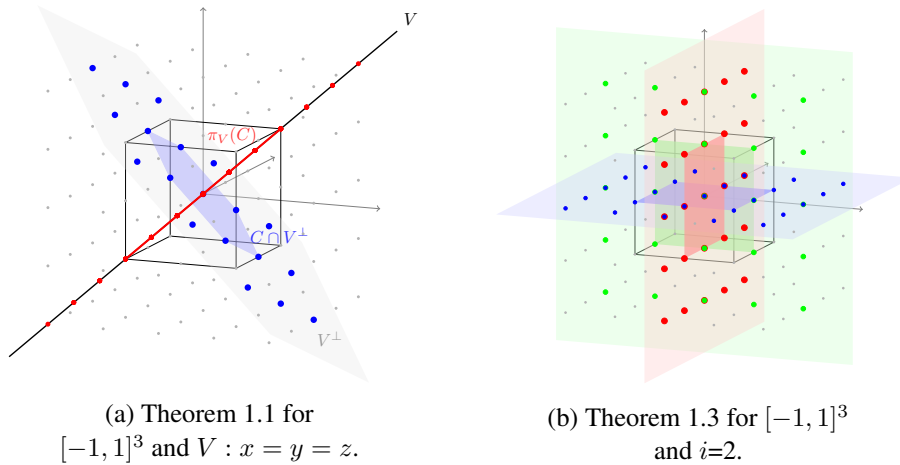

\begin{theorem}\label{thm: ub proj}
Let $K\subseteq\R^d$ be a convex body containing the origin, $\Lambda\subseteq \R^d$ a lattice and $V\subseteq \R^d$ a rational linear subspace, with $\dim V=\ell$ and $i\in[d]$. If $\pi_V$ denotes the orthogonal projection of $\R^d$ to $V$, the following inequality holds:
\begin{align}\label{eq:ub proj}
\mu_i(K, \Lambda)\leqslant \max_{\substack{0 \le j \le \ell \\ 0 \le i-j \le d-\ell}} \mu_j(\pi_V(K), \pi_V(\Lambda))+ \mu_{i-j}(K\cap V^\perp, \Lambda \cap V^\perp)
\end{align}
\end{theorem}
This generalizes a fundamental lemma appearing in \cite{CodenottiSantosSchymura}, and further used in \cite{MalikiosisSantosSchymura}. The same lemma was also used by Reis and Rothvoss \cite{ReisRothvoss} in the proof of their breakthrough result, which implies a near optimal bound on the flatness constant.

This bound depends on covering minima of the projection of the convex body and lattice onto a fixed rational subspace and the intersection with the orthogonal complement of the given rational subspace, as seen in Figure \ref{fig: bounds} (a).

The structure of the bound in Theorem \ref{thm: ub proj} is particularly suitable for direct sums of convex bodies, and we show that it is indeed sharp for direct sums. In particular, we derive the following result which describes all the covering minima of a direct sum in terms of covering minima of its summands, thus unifying the results mentioned above on the covering radius and lattice width of the direct sum.

\begin{theorem}\label{thm: direct sum}
Let $\R^d=V\oplus W$, $\dim(V)=\ell$, where $\dim(W)=d-\ell$. Let $K\subseteq V$ and $L\subseteq W$ be convex bodies that contain the origin, and $\Lambda \subseteq V$ and $\Gamma \subseteq W$ lattices. 
Then, for every and $i\in [d]$, the following equality holds:
 $$\mu_i(K\oplus L, \Lambda\oplus\Gamma)=\max_{\substack{0 \le j \le \ell \\ 0 \le i-j \le d-\ell}} \mu_j(K, \Lambda)+\mu_{i-j}(L,\Gamma).$$
\end{theorem}

In contrast to the bound presented in Theorem \ref{thm: ub proj}, our next result depends only on finitely many values of certain covering radii. Specifically, it is the maximum of the covering radii of the intersection of the convex body and lattice with the coordinate hyperplanes of dimension corresponding to the index of the covering minimum, as seen in Figure \ref{fig: bounds} (b).
\begin{theorem}\label{thm: ub intersec}
Let $K\subseteq \R^d$ be a convex body, $\Lambda \subseteq \R^d$ a lattice and let $\{ f_1, \dots f_d\}$ be a basis of $\Lambda$. For $I\in {[d]\choose i}$, denote by $L_I=\operatorname{span}_\R\{f_i\; |\; i\in I \}$ the $i$-dimensional linear subspace of $\R^d$ corresponding to $I$ and the given basis. If for every $I\in {[d]\choose i}$, $\dim(K\cap L_I)=i$, then:
$$\mu_i(K, \Lambda)\leqslant \max\left\{ \mu(K\cap L_I, \Lambda \cap L_I)\;| \; I\in {[d]\choose i} \right\}.$$
\end{theorem}

A bound of this form is of great interest, since the covering radius is a much better understood functional than the covering minima. Additionally, efficient algorithms for calculating the covering radius of a rational polytope exist, see \cite{DavidPacoPaco} and \cite{Cslovjecsek}, whereas to our knowledge, efficient algorithms for general covering minima are not known.

We extend the equivalence between the conjectures in \cite{MerinoSchymura} and \cite{CodenottiSantosSchymura}with a third statement. Namely, we show that the conjectured values of covering minima of standard terminal simplices should coincide with those of terminal polytopes.
\begin{proposition}\label{thm: big equivalence}
For every $d\in \N$, the following are equivalent:
\begin{enumerate}[label=\roman*)]
\item $\mu(P)\leqslant\frac{i}{2}$ for every $i\leqslant d$ and every non-hollow lattice $i$-polytope $P$.

\item $\mu_i(T_n)=\frac{i}{2}$ for every $n\geq d$ and every $i\leqslant d$.

\item $\mu_i(T)=\frac{i}{2}$ for every $n\geq d$, every terminal $n$-polytope $T$ and every $i\leqslant d$.

\end{enumerate}
\end{proposition}

\noindent We further give two new upper bounds on the covering minima of standard terminal simplices.
\subsection{Organization of the paper}

The proofs of Theorems \ref{thm: ub proj} and \ref{thm: direct sum} can be found in Section \ref{sec:projection}.
The proof of Theorem \ref{thm: ub intersec} appears in Section \ref{sec:intersection}. We further show that the bound in Theorem \ref{thm: ub intersec} is sharp for \textit{locally anti-blocking bodies}.
In Section \ref{sec: conjectures}, we give an overview of the conjectures of Codenotti, Santos and Schymura \cite{CodenottiSantosSchymura}, and show Proposition \ref{thm: big equivalence}.
Lastly, in Section \ref{sec:terminal} we apply our bounds to terminal simplices and certain weighted generalizations, and put these calculations into the context of Section \ref{sec: conjectures}.

\section{Projection Bound on Covering Minima}\label{sec:projection}
In this section, we prove Theorems \ref{thm: ub proj} and \ref{thm: direct sum}.

\begin{notation}
For a convex body $K\subseteq \R^d$ and lattice $\Lambda\subseteq \R^d$, we denote $\mu_0(K,\Lambda):=0$ for convenience. Notice that this agrees with the definition of covering minima, since indeed $0K+\Lambda=\{\mathbb{0}_d\}+\Lambda=\Lambda$ intersects $\R^d$, which is the only $d$-dimensional affine subspace of $\R^d$.
\end{notation}

In general, to prove that some $\mu>0$ is an upper bound for $\mu_i(K,\Lambda)$, one needs to show that $\mu K + \Lambda$ intersects every $(d-i)$-dimensional affine subspace of $\R^d$. The structure of the bound \eqref{eq:ub proj} splits into two parts on the right hand side -- one coming from the projection onto the chosen rational linear subspace, and the other from the intersection with its orthogonal complement.  The maximum in the bound comes from the differing dimensions of the corresponding projections and intersections of arbitrary $(d-i)$-dimensional affine subspaces.
\begin{proof}[Proof of Theorem \ref{thm: ub proj}]
Let $x+U$ be an arbitrary $(d-i)$-dimensional affine subspace of $\R^d$, where $x\in \R^d$ and $U\subseteq \R^d$ is a linear subspace. It suffices to show that $x+U$ intersects $(\mu_j(\pi_V(K), \pi_V(\Lambda))+ \mu_{i-j}(K\cap V^\perp, \Lambda \cap V^\perp))K + \Lambda $, for some $0 \le j \le \ell,\; 0 \le i-j \le d-\ell$.

Let $\dim(\pi_V(U))=\ell-j$ and $\dim(U\cap V^\perp)=(d-i)-(\ell-j)=(d-\ell)-(i-j)$. For brievity, let $\mu_j:=\mu_j(\pi_V(K),\pi_V(\Lambda))$ and $\mu_{i-j}:=\mu_{i-j}(K\cap V^\perp, \Lambda \cap V^\perp)$.

Since $\pi_V(x)+\pi_V(U)$ is a $(\ell-j)$-dimensional affine subspace of $V$, there exist $u_1\in U$, $p\in K$ and $a\in \Lambda$ such that 
\begin{align}\label{eq:cmj}
\pi_V(x)+\pi_V(u_1)=\mu_j \pi_V(p)+\pi_V(a)
\end{align}

Denote by $y=\pi_{V^\perp}(x)+\pi_{V^\perp}(u_1)-\mu_j\cdot \pi_{V^\perp}(p)-\pi_{V^\perp}(a)\in V^\perp$.

\noindent Notice that $y\in V^\perp$ implies that $\dim((y+U)\cap V^\perp)=\dim(U\cap V^\perp)=(d-\ell)-(i-j)$. 

\noindent Since $(y+U)\cap V^\perp$ is a $((d-\ell)-(i-j))$-dimensional affine subspace of $V^\perp$, there exist $u_2 \in U\cap V^\perp$, $q\in K\cap V^\perp$ and $b\in \Lambda \cap V^\perp$ such that

$$y+u_2=\mu_{i-j}q+b$$
$$\Rightarrow \pi_{V^\perp}(x)+\pi_{V^\perp}(u_1)-\mu_j\cdot \pi_{V^\perp}(p)-\pi_{V^\perp}(a)+u_2=\mu_{i-j}\cdot q+b$$
\begin{align}\label{eq:cmi-j}
\Rightarrow \pi_{V^\perp}(x)+\pi_{V^\perp}(u_1)=\mu_j\cdot \pi_{V^\perp}(p)+\mu_{i-j}\cdot q+\pi_{V^\perp}(a)+b
\end{align}
Now, adding the equations \ref{eq:cmj} and \ref{eq:cmi-j} we get:

$$ x+u_1+u_2=\mu_j\cdot p+\mu_{i-j}\cdot q+a+b$$
$$\Rightarrow x+u_1+u_2=(\mu_j+\mu_{i-j})\left(\frac{\mu_j}{\mu_j+\mu_{i-j}}p+\frac{\mu_{i-j}}{\mu{j}+\mu_{i-j}}q\right)+a+b.$$

This point verifies that $x+U$ intersects $(\mu_j+\mu_{i-j})K + \Lambda$. Since $j$ depends on $U$, and by construction $j$ can take any value for which $0 \le j \le \ell,\; 0 \le i-j \le d-\ell$, it follows that $\mu_i(K,\Lambda)$ is bounded from above by the maximum of $\mu_j+\mu_{i-j}$ for all $j$ in this index set. \end{proof}

Our next goal is to show that the bound we have just proved is sharp when utilized for a direct sum of two convex bodies, with the chosen rational linear subspace and its orthogonal complement being exactly the decomposition of $\R^d$ that corresponds to the direct sum.
\begin{proof}[Proof of Theorem \ref{thm: direct sum}]
The upper bound is a direct consequence of Theorem \ref{thm: ub proj}, since $\pi_V(K\oplus L)=K$, $\pi_V(\Lambda\oplus\Gamma)=\Lambda$, $K\oplus L \cap W= L$ and $\Lambda\oplus \Gamma\cap W =\Gamma$.

For the lower bound, it suffices to show that for every $j$ such that $0\leqslant j \leqslant l$ and $0\leqslant i-j\leqslant d-l$, 
$$\mu_i(K\oplus L, \Lambda\oplus\Gamma)\geq \mu_j(K,\Lambda)+\mu_{i-j}(L, \Gamma)$$
holds.  

Notice that $K\oplus L$ and $\Lambda\oplus\Gamma$ project to $K$ and $\Lambda$, therefore $\mu_i(K\oplus L, \Lambda\oplus\Gamma)\geq \mu_i(K,\Lambda)$, which corresponds to the case $j=i$. Similarly, for $j=0$ we see $\mu_i(K\oplus L, \Lambda\oplus\Gamma)\geq \mu_i(L,\Gamma)$.

Assume $j, i-j \neq 0$, moreover, $\mu_j:=\mu_j(K, \Lambda)>0$ and $\mu_{i-j}:=\mu_{i-j}(L,\Gamma)>0$.

Suppose that $\mu_i(K\oplus L, \Lambda\oplus\Gamma)< \mu_j+\mu_{i-j}$, and take $0<c<\mu_j$ and $0<c'<\mu_{i-j}$ such that $c+c'=\mu_i(K\oplus L, \Lambda\oplus\Gamma)$. Then there exists a $(l-j)$-dimensional linear subspace $U_V\leqslant V$ and $x\in V$ such that 
$(x+U_V)\cap cK=\emptyset$ and a $(d-l-i+j)$-dimensional linear subspace $U_W\leqslant W$ and $y\in W$ such that $(y+U_W)\cap c'L=\emptyset$.

Then $\dim(U_V\oplus U_W)=(d-i)$, so $x+y+(U_V\oplus U_W)$ intersects $(c+c')(K\oplus L) + (\Lambda\oplus\Gamma)$. Therefore, there exist $v\in U_V$, $w\in U_W$, $p\in K$, $q\in L$, $\lambda\in [0,1]$, $a\in\Lambda$ and $b\in \Gamma$ such that
$$x+y+v+w=(c+c')(\lambda p+(1-\lambda)q)+a+b.$$ 
\noindent The sum is direct, therefore $x+v=(c+c')\lambda p+a$ and $y+w=(c+c')(1-\lambda)q+b$. Because $\mathbb{0}_d\in K$ and $\mathbb{0}_d\in L$, $\lambda p\in K$ and $(1-\lambda)q\in L$. Moreover, since $(x+U_V)\cap cK=\emptyset$ and $(y+U_W)\cap c'L=\emptyset$, we can conclude $(c+c')\lambda>c$ and $(c+c')(1-\lambda)>c'$. This is equivalent to $c'\lambda>c(1-\lambda)$ and $c(1-\lambda)>c'\lambda$, which is not plausible.
\end{proof}

The covering minima of the unimodular simplex were calculated in \cite{KannanLovasz} using their upper bound depending on successive minima of the difference body. Moreover, the covering minima of the standard crosspolytope were calculated in \cite{MerinoSchymura}. One usecase of Theorem \ref{thm: direct sum} is a simple calculation of all covering minima of direct sums of segments, which is a family of polytopes containing both the unimodular simplex and the standard crosspolytope.

\begin{example} Let $a_j<b_j$ and $b_j-a_j\leqslant b_{j-1}-a_{j-1}$. The only non-zero covering minima of segments are their covering radii $\mu_1$, which are equal to the reciprocal of the lenght of the given segment. Using Theorem \ref{thm: direct sum} to calculate the $i$-th covering minimum of the direct sum of the segments $[a_j,b_j]$, we get:
$$\mu_i\left(\bigoplus\limits_{j=1}^d [a_je_j, b_je_j]\right) = \sum_{j=1}^{i}\frac{1}{b_j-a_j}.$$
Plugging in $a_j=0$ and $b_j=1$, we get that the $i$-th covering minimum of the standard unimodular simplex is $i$.\\
Now, plugging in $a_j=-1$ and $b_j=1$, we get that the $i$-th covering minimum of the standard crosspolytope is $\frac{i}{2}$.
\end{example}

\section{Intersection Bound on Covering Minima}\label{sec:intersection}

The goal of this section is to bound $\mu_i(K)$ from above with a maximum of finitely many covering radii. It is easy to check via dimension arguments that every $(d-i)$-dimensional affine subspace of $\R^d$ intersects at least one $i$-dimensional coordinate subspace. Therefore, if a dilate of $K$ covers all $i$-dimensional coordinate subspaces, it is at least the $i$-th covering minimum.\\

\begin{notation} For $I\in {[d]\choose i}$, denote by $L_I=\operatorname{span}_\R\{f_i\; |\; i\in I \}$ the $i$-dimensional linear subspace of $\R^d$ corresponding to $I$ and the given basis $\{f_1,\dots, f_d\}$ of the lattice $\Lambda$. When the lattice is $\Z^d$, we will assume it's the standard basis $\{e_1,\dots,e_d\}$.
\end{notation}

\begin{proof}[Proof of Theorem \ref{thm: ub intersec}]
Let $U$ be an arbitrary $(d-i)$-dimensional affine subspace of $\R^d$, and let $I\in {[d]\choose i}$ be such that $U\cap L_I\neq \emptyset$. Restricting to $L_I$, since $K\cap L_I$ is full dimensional and $\Lambda\cap L_I$ is a lattice in $L_I$ since $f_1,\dots,f_d$ is a lattice basis, by the definition of covering radius:
$$\mu(K\cap L_I, \Lambda \cap L_I) (K\cap L_I) + \Lambda\cap L_I = L_I$$
$$\Rightarrow \mu(K\cap L_I, \Lambda \cap L_I)K+\Lambda \supseteq L_I$$
$$\Rightarrow (\mu(K\cap L_I, \Lambda \cap L_I)K+\Lambda) \cap U \supseteq L_I\cap U \neq \emptyset.$$

\noindent Let $t\geq 0$. The previous calculations imply that for $(tK+\Lambda)\cap U$ to be nonempty, it is enough for $t$ to be greater or equal to $\mu(K\cap L_I, \Lambda \cap L_I)$. Therefore, for $tK+\Lambda$ to intersect all $(d-i)$-dimensional affine subspaces, it is enough for $t$ to beT the maximum of these values when we iterate $I\in {[d] \choose i}$. 
\end{proof}

Notice that the supposition of the intersection of the convex body with all coordinate $i$ subspaces is not too demanding -- requiring that $K$ contains the origin in its interior would imply it, and covering minima are translation invariant. However, the sections of the body are clearly not translation invariant, therefore this bound has different values for different translations of the same convex body.

Our next goal is to identify a significant family of convex bodies for which the bound in Theorem \ref{thm: ub intersec} is sharp.

The following family of \textit{locally anti-blocking bodies} was defined in \cite{DefLAB} as a generalization of both \textit{anti-blocking} and \textit{unconditional} convex bodies.

\begin{definition}
A convex body $K\subseteq \R^d$ is \textit{locally anti-blocking} if for every coordinate subspace, the projection and intersection of the convex body with respect to this subspace coincide. That is, $
K\cap L_I=\pi_{L_I}(K)$ for all $I\subseteq [d]$.

\noindent A locally anti-blocking body is \textit{proper} if it contains the origin in its interior.
\end{definition}

Recall that the covering minima can be seen as the maximal covering radius of the projection of the given convex body over all rational linear subspaces of the appropriate dimension, and more specifically: 
 $$\mu_i(K, \Lambda)=\max\left\{ \mu(\pi(K), \pi(\Lambda ))\;| \; \pi \text{ is a rational linear projection of rank }i\right\}. $$

Therefore, for every rational linear subspace $V\subseteq \R^d$ of dimension $i$, $\mu_i(K, \Lambda)\geqslant \mu(\pi_V(K), \pi_V(\Lambda))$. Since for locally anti-blocking bodies the projections and intersections with coordinate subspaces coincide,  the maximizer of the covering radius of a coordinate intersection in the upper bound given by Theorem \ref{thm: ub intersec} is also a lower bound for the covering minimum of this convex body. These observations prove the following corollary:

\begin{corollary}\label{cor: locally anti-blocking}
Let $K\subseteq \R^d$ be a proper locally anti-blocking body and $i\in [d]$. Then:
$$\mu_i(K)= \max\left\{ \mu(K\cap L_I)\;| \; I\in {[d]\choose i} \right\}.$$
\end{corollary}

This corollary is particularly interesting from the viewpoint of explicit calculations, since it reduces all the explicitly calculated values of all covering minima that exist to the authors knowledge to calculations of covering radii. More specifically, the unimodular simplex \cite{KannanLovasz} and the family $P_{d,k}=C_d\cap kC_d^*$ \cite{MerinoSchymura}.

\begin{remark}
Corollary \ref{cor: locally anti-blocking} implies that for a proper locally anti-blocking body $K$, $\mu_i(K)+\Z^d$ contains all coordinate subspaces of dimension $i$. 
This is the condition in Theorem 3.2 (ii) of \cite{MerinoSchymura}, which gives a lower bound of
$$\mu_i(K)^d\operatorname{vol}(K)\geqslant \frac{{i!}^{\frac{d}{i}}}{d!}.$$
\end{remark}

\section{Covering radius of non-hollow lattice polytopes}\label{sec: conjectures}

Recall that the \textit{terminal simplex} in dimension $d$ with respect to the standard integer lattice is $T_d:=\operatorname{conv}(-\mathbb{1}_d, e_1,\dots, e_d)$. It is a lattice simplex with the origin as its only interior lattice point, and shows up in various contexts as a natural candidate for a non-hollow version of the unimodular simplex.  We call a full dimensional lattice polytope a \textit{terminal polytope} if it can be represented as a direct sum of translates of terminal simplices. The family of terminal polytopes interpolates between the terminal simplex and the standard crosspolytope, which can be seen as a direct sum of segments $[-e_j, e_j] $, that is, terminal simplices of dimension $1$.
 
Codenotti, Santos and Schymura \cite{CodenottiSantosSchymura} pose the question of what the maximal covering radius of a non-hollow lattice polytope is, and conjecture the following which they prove up to dimension 3.

\begin{conjecture}\label{conj A}
\cite[Conj. A]{CodenottiSantosSchymura}
Let $P\subseteq \R^d$ be a non-hollow lattice polytope. Then
$$\mu(P)\leqslant \frac{d}{2},$$
where equality holds if and only if $P$ is a terminal $d$-polytope up to a unimodular transformation.
\end{conjecture}

Regarding the covering minima of terminal simplices, Gonz\' alez Merino and Schymura \cite{MerinoSchymura} conjecture the following:
\begin{conjecture}\cite[Rem. 4.9]{MerinoSchymura} \label{conj B}
For every $d\in \N$ and $i\in [d]$,
$$\mu_i(T_d)=\frac{i}{2}.$$
\end{conjecture}

They prove the case $i=d$ \cite[Prop. 4.8]{MerinoSchymura}, and one can easily check that the lattice width of a standard terminal simplex is $2$, and therefore $\mu_1(T_d)=\frac{1}{2}$.

These two conjectures are proven to be equivalent in the following manner:

\begin{theorem}\cite[Thm. 1.2]{CodenottiSantosSchymura}
\label{conjecture equivalence}
For every $d\in \N$, the following are equivalent:
\begin{enumerate}[label=\roman*)]
\item $\mu(P)\leqslant\frac{i}{2}$ for every $i\leqslant d$ and every non-hollow lattice $i$-polytope $P$.

\item $\mu_i(T_n)=\frac{i}{2}$ for every $n\geqslant d$ and every $i\leqslant d$.

\end{enumerate}
\end{theorem}

Taking our result on covering minima of direct sums from Theorem \ref{thm: direct sum} into consideration, we can now see that if Conjecture \ref{conj B} holds, then not only will the terminal simplex be the maximizer of all covering minima amongst non-hollow lattice polytopes, but so will all terminal polytopes. 

\begin{corollary}\label{terminal polytope}
Let $T$ be a terminal $d$-polytope. If we assume that Conjecture \ref{conj B} holds in dimensions up to $d$, then for all $i\in [d]$,
$$\mu_i(T)=\frac{i}{2}.$$
\end{corollary}

\begin{proof}
Let $k\in \N$, $l_1,\dots,\l_k\in \N$ s. t. $\sum\limits_{i=1}^k l_i=d$ and for $j\in[k]$, $u_j\in \R^{k_j}$ such that $\mathbb{0}_{l_j}\in u_j+S(\mathbb{1}_{l_j+1})$. An arbitrary terminal $d$-polytope can be seen as $T=(u_1+S(\mathbb{1}_{l_1+1}))\oplus\dots\oplus(u_k+S(\mathbb{1}_{l_k+1}))$, and this decomposition into a direct sum agrees with the decomposition of $\R^d$ into $\R^d=\R^{l_1}\oplus\dots\oplus\R^{l_k}$.

Since we assume that Conjecture \ref{conj B} holds in dimensions up to $d$, we know that for every $j\in [k]$ and every $0\leqslant s\leqslant l_j$, $\mu_s(u_j+S(\mathbb{1}_{l_j+1}))=\frac{s}{2}$.

For any $i\in [d]$, using this and applying Theorem \ref{thm: direct sum} we get:

$$\mu_i(T)=\max\left\{\sum\limits_{j=1}^k \frac{s_j}{2} \;\vline\; 0\leqslant s_j\leqslant l_j,\; \sum\limits_{j=1}^k s_j=i\right\},$$
that is, $ \mu_i(T)=\frac{i}{2}$.

\end{proof}

Combining Theorems \ref{conjecture equivalence} and  \ref{terminal polytope} gives Proposition \ref{thm: big equivalence}.

\section{Upper Bounds on Covering Minima of (Weighted) Terminal Simplices}\label{sec:terminal}

The goal of this section is to apply the upper bounds obtained in Sections \ref{sec:projection} and \ref{sec:intersection} to terminal simplices, in view of the conjectures mentioned in Section \ref{sec: conjectures}. 

To apply Theorems \ref{thm: ub proj} and \ref{thm: ub intersec} to terminal simplices, we will need the following family of weighted versions of thereof, which was introduced in \cite{CodenottiSantosSchymura}.

\begin{definition}
Let $\omega=(\omega_0,\dots, \omega_d)\in \R^{d+1}_{> 0}$ be a vector of weights. We define the following family of simplices:
$$S(\omega):=\operatorname{conv}(-\omega_0 \mathbb{1}_d, \omega_1 \cdot e_1, \dots, \omega_d\cdot e_d)\subseteq \R^d.$$
Specifically,  $S(\mathbb{1}_{d+1})=T_d$.
\end{definition}

The authors further calculate the covering radius of all weighted terminal simplices, and conjecture the values of the other covering minima, proving also the case $i=1$.

\begin{theorem}\label{thm: cr simplices}
{\cite[Theorem 1.4]{CodenottiSantosSchymura}}
For every $\omega \in \R^{d+1}_{>0}$, we have:
$$\mu(S(\omega))=\frac{\sum\limits_{0\leqslant i<j\leqslant d}\frac{1}{\omega_i \omega_j}}{\sum\limits_{i=0}^d \frac{1}{\omega_i}}.$$
\end{theorem}

\begin{conjecture}\cite[Conj. 5.3]{CodenottiSantosSchymura} \label{conj: cm weighted}
For every $\omega\in \R^{d+1}_{>0}$ with $\omega_0\leqslant\dots\leqslant \omega_d$, and every $i\in [d]$, the $i$-th covering minimum of $S(\omega)$ is attained by the projection to the first $i$ coordinates, that is:

$$\mu_i(S(\omega))=\frac{\sum\limits_{0\leqslant j<k\leqslant i}\frac{1}{\omega_j \omega_k}}{\sum\limits_{j=0}^i \frac{1}{\omega_j}}.$$
\end{conjecture}

As in Conjecture \ref{conj B}, which is a special case of this conjecture, the lower bound can clearly be seen by noticing that $\pi_{L_I}(S(\omega))=S(\omega_0, (\omega_j)_{j\in I})$ and applying Theorem \ref{thm: cr simplices}. However, no explicit upper bounds are known to the author, other than the one coming from monotonicity of covering minima.

For the purpose of bounding the covering minima of $S(\omega)$ and specifically $T_d$ using Theorems  \ref{thm: ub proj} and \ref{thm: ub intersec}, we additionally need to calculate the intersections of $S(\omega)$ with respect to arbitrary coordinate subspaces, which is the purpose of the following lemma.

\begin{lemma} \label{lem: weighted}
Let $i\in [d]$, $I\in {{[d]}\choose i}$, and $\omega\in \R^{d+1}_{>0}$. Then,
$$
S(\omega)\cap L_I
=
S\left(
\frac{1}{\sum\limits_{\substack{k \notin I}} \frac{1}{\omega_k}},\;
(\omega_j)_{j\in I}
\right)\subseteq \R^I,
$$

\noindent where the equality means set equality when restricting the left hand side set to the coordinates in $I$.
\end{lemma}

\begin{proof}
A point $x\in \R^d$ is in the set $S(\omega)\cap L_I$ if and only if it can be represented as $x=-\lambda_0\omega_0\mathbb{1}_d+\sum\limits_{j=0}^d \lambda_j\omega_je_j$, where $\lambda_j\geq 0$ for all $0\leqslant j \leqslant d$, $\sum\limits_{j=0}^d\lambda_j=1$ and for all $k\notin I$,  $-\lambda_0\omega_0+\lambda_k\omega_k=0$ holds. The last condition can be rewriten as $\lambda_k=\frac{\omega_0}{\omega_k}\lambda_0$ for all $k\notin I$. Notice that since $\omega\in \R^{d+1}_>0$, the non-negativity of $\lambda_0$ implies the non-negativity of $\lambda_k$ for $k\notin I$. Now we rewrite $1=\sum\limits_{j=0}^d\lambda_j=\lambda_0\omega_0\sum\limits_{\substack{k \notin I}}\frac{1}{\omega_k}+\sum\limits_{j\in I}\lambda_j$.

Taking this into consideration, we rewrite 
$$x=-\lambda_0\omega_0\mathbb{1}_I+\sum\limits_{j\in I}\lambda_j\omega_je_j =-\left(\lambda_0\omega_0\sum\limits_{\substack{ k \notin I}}\frac{1}{\omega_k}\right)\cdot \frac{1}{\sum\limits_{\substack{ k \notin I}} \frac{1}{\omega_k}}\mathbb{1}_I+\sum\limits_{j\in I}\lambda_j\omega_je_j,$$

that is, $x\in S\left(
\frac{1}{\sum\limits_{ k \notin I} \frac{1}{\omega_k}},\;
(\omega_j)_{j\in I}
\right)$.

\end{proof}

As a direct corollary of this lemma when setting $\omega=\mathbb{1}_{d+1}$, we get the following statement.

\begin{corollary} \label{lem: terminal simplex intersected with coordinate subspace}
Let $i\in [d]$ and let $I\in {[d]\choose i}$. Then, $T_d\cap L_I=S((\frac{1}{d-i+1}, 1, \dots, 1))$, where the weight vector has $i+1$ entries, and equality means really the equality if we restrict to the coordinates in $I$.
\end{corollary}

Bounding the values of weighted terminal simplices is a generalization of the problem of bounding those of terminal simplices. Taking this and Corollary \ref{lem: terminal simplex intersected with coordinate subspace} into consideration, we see that Theorem  \ref{thm: ub proj} can be most explicitly used when intersecting with a linear subspace of dimension $1$ or $2$. In these situations, all the covering minima of the weighted terminal simplices obtained as the intersection of the terminal simplex with the coordinate subspaces are either the covering radius or the reciprocal of the lattice width, and therefore known values, since Conjecture \ref{conj: cm weighted} is known to hold for $i=1$ and $i=d$. 

The following is an explicit upper bound for $\mu_i(T_d)$, obtained by successive use of Theorem \ref{thm: ub proj} for a linear subspace $V$ of codimension $1$.

\begin{corollary}\label{cor: UB proj terminal simplex}
For every $d\in \N$ and $2\leqslant i \leqslant d$, 
$$\mu_i(T_d)\leqslant\frac{1}{2}+\sum\limits_{j=0}^{i-2} \frac{d-j}{d-j+1}.$$
\end{corollary}
\begin{proof}
Observe the coordinate hyperplane $L=\R^{d-1}\times \{0\}$. Then, $\pi_L(T_d)=T_{d-1}\times \{ 0\}$, and as seen in Lemma  \ref{lem: terminal simplex intersected with coordinate subspace}, $T_d\cap L^\perp = \{ \mathbb{0}_{d-1}\}\times S(\frac{1}{d}, 1)=\{ \mathbb{0}_{d-1}\}\times \operatorname{conv}(-\frac{1}{d}e_d, e_d)$. Since the ambient space does not matter for covering minima purposes, and $\pi_L(\Z^d)=\Z^{d-1}\times\{0\}$ and $\Z^d\cap L^\perp=\{\mathbb{0}_{d-1}\}\times \Z$, we can just see these as $T_{d-1}$ and $[-\frac{1}{d}, 1]$ in corresponding standard lattices. Notice also that $\mu([-\frac{1}{d}, 1])=\frac{d}{d+1}$, since the covering radius of every segment is just the scaling needed to get its lenght to be $1$. Then, Theorem \ref{thm: ub proj} applied to $T_d$ and $L$ gives the following:
$$\mu_i(T_d)\leqslant \max_k \left (\mu_k(T_{d-1})+\mu_{i-k}([-\frac{1}{d}, 1])\right)=\max \left\{ \mu_i(T_{d-1}), \mu_{i-1}(T_{d-1})+\frac{d}{d+1}\right\}. $$
Notice that on the right hand side, the dimension of the terminal simplices observed has dropped.
By successive application of this inequality, we can get to one of the values that we know -- the first covering minimum of a terminal simplex being $\frac{1}{2}$, or the covering radius of a terminal simplex being half of its dimension. Specifically, by applying this inequality to the first element in the set we're taking the maximum of, we get:
\begin{align} \label{eq: in pf proj}
\mu_i(T_d)\leqslant \max \left\{ \mu_i(T_{d-2}), \mu_{i-1}(T_{d-2})+\frac{d-1}{d}, \mu_{i-1}(T_{d-1})+\frac{d}{d+1}\right\}.
\end{align}

 Let's first prove $\mu_i(T_d)\leqslant \max \{\frac{i}{2}, \mu_{i-1}(T_{d-1})+\frac{d}{d+1}\}$. 

Since $T_{d-1}$ projects to $T_{d-2}$ when projecting out the last coordinate, $\mu_{i-1}(T_{d-2})\leqslant \mu_{i-1}(T_{d-1}).$ Additionally, $\frac{d-1}{d}<\frac{d}{d+1}$, therefore $\mu_{i-1}(T_{d-2})+\frac{d-1}{d}<\mu_{i-1}(T_{d-1})+\frac{d}{d+1}$, and the former can be removed from the set we are maximizing over in \eqref{eq: in pf proj}. This brings us to
 $\mu_i(T_d)\leqslant \max \{ \mu_i(T_{d-2}), \mu_{i-1}(T_{d-1})+\frac{d}{d+1}\}$. Notice that we just dropped the dimension of the terminal simplex by 1 again in the first element of the set we're maximizing. Repeating this process $d-i$ times in total, we get to the covering radius of a terminal simplex, which is a value we know:
$$\mu_i(T_d)\leqslant \max \left\{ \mu_i(T_i), \mu_{i-1}(T_{d-1})+\frac{d}{d+1}\right\}=\max \left\{ \frac{i}{2}, \mu_{i-1}(T_{d-1})+\frac{d}{d+1}\right\}.$$ 

Applying this inequality to the term $\mu_{i-1}(T_{d-1})$ on the right hand side, we get:
$$\mu_i(T_d)\leqslant \max \left\{ \frac{i}{2}, \frac{i-1}{2}+\frac{d}{d+1}, \mu_{i-2}(T_{d-2})+\frac{d-1}{d}+\frac{d}{d+1}\right\}.$$ 
Now, notice that since $d\geq 1$, we have $\frac{d}{d+1}\geq \frac{1}{2}$, i.e., $\frac{i-1}{2}+\frac{d}{d+1}\geq \frac{i}{2}$. Moreover, for every $0\leqslant j\leqslant i$, since $d>i$, we have $\frac{d-j}{d-j+1}\geq \frac{1}{2}$. Therefore, applying the inequality $\mu_{i-j}(T_{d-j})\leqslant \max \{ \frac{i-j}{2}, \mu_{i-j-1}(T_{d-j-1})+\frac{d-j}{d-j+1}\}$ successively $i-1$ times for $0\leqslant j \leqslant i-2$ brings us to:
$$\mu_i(T_d)\leqslant \max \left\{\frac{1}{2}+\sum\limits_{j=0}^{i-2} \frac{d-j}{d-j+1}, \mu_1(T_{d-i+1})+ \sum\limits_{j=0}^{i-2} \frac{d-j}{d-j+1}\right\}=\frac{1}{2}+\sum\limits_{j=0}^{i-2} \frac{d-j}{d-j+1}.$$
\end{proof}

\begin{remark}
Note that one could do similar calculations by starting with the terminal simplex and  a coordinate subspace of codimension $2$, as well as starting with a weighted terminal simplex and a coordinate hyperplane.
\end{remark}

The following proposition is the application of the upper bound in Theorem \ref{thm: ub intersec} to an arbitrary weighted terminal simplex, with the goal of approaching the wanted upper bound from Conjecture \ref{conj: cm weighted}.

\begin{proposition}
Let $\omega\in \R^{d+1}_{>0}$ such that $\omega_0\leqslant \dots \leqslant \omega_d$. Then, for every $i<d$:
\[
\mu_i(S(\omega)) \leqslant
\frac{
\left(\dfrac{1}{\omega_0}+\sum\limits_{k=i+1}^{d}\dfrac{1}{\omega_k}\right)
\left(\sum\limits_{j=1}^{i}\dfrac{1}{\omega_j}\right)
+\sum\limits_{\substack{1\le s<t\le i}}\dfrac{1}{\omega_s\omega_t}
}{
\sum\limits_{k=0}^{d}\dfrac{1}{\omega_k}
}.
\]

\end{proposition}

\begin{proof}
Combining Theorem \ref{thm: ub intersec} with Lemma \ref{lem: weighted}, we get
$$\mu_i(S(\omega))\leqslant \max \left\{ \mu\left(S\left(
\frac{1}{\sum\limits_{\substack{k \notin I}} \frac{1}{\omega_k}},\;
(\omega_j)_{j\in I}
\right)\right)\; | \; I \in {[d]\choose i}\right\}.$$

The covering minima on the right hand side are given by Theorem \ref{thm: cr simplices} to be

$$\mu\left(S\left(
\frac{1}{\sum\limits_{\substack{k \notin I}} \frac{1}{\omega_k}},\;
(\omega_j)_{j\in I}
\right)\right)=  \frac{\left(\sum\limits_{k\notin I}\frac{1}{\omega_k}\right)\left(\sum\limits_{j\in I} \frac{1}{\omega_j}\right)+\sum\limits_{\substack{s<t\\s,t\in I}}\frac{1}{\omega_s\omega_t}}
{\sum\limits_{k=0}^d \frac{1}{\omega_k}}.$$

To prove the proposition, we need to show that this expression is maximized when $I=[i]$. Denote by $C=\sum\limits_{k=0}^d \frac{1}{\omega_k}$, $L(I)=\sum\limits_{j\in I} \frac{1}{\omega_j}$ and $Q(I)=\sum\limits_{j\in I} \frac{1}{\omega_j^2}$. We need to maximize $\frac{(C-L(I))L(I)+\frac{1}{2}(L(I)^2-Q(I))}{C}$, which reduces to maximizing 
$$F(I)=2C\cdot L(I)-L(I)^2-Q(I),$$
 since $C$ is a constant.

Suppose that $I\neq [i]$. Since $i<d$, there exists an $a\in I$ and $b\in [d]\setminus I$ such that $b<a$, and denote $I'=I\setminus \{ a\} \cup \{b\}$. Let $\delta=\frac{1}{\omega_b}-\frac{1}{\omega_a}$, which is non-negative since $\omega_b\leqslant\omega_a$. \\

Then $L(I')=L(I)+\delta$ and $Q(I')=Q(I)+\delta^2+\delta\frac{1}{\omega_a}$.

$$F(I')-F(I)=2C\delta-2L(I)\delta-\delta^2-\delta^2-\delta\frac{1}{\omega_a}=$$
$$=2\delta\left(C-L(I)-\delta-\frac{1}{\omega_a}\right)=2\delta\left(C-L(I)-\frac{1}{\omega_b}\right).$$

Recalling the definitions of $C$ and $L(I)$, we see $F(I')-F(I)=2\delta \sum\limits_{j\notin I\cup \{ b\}} \frac{1}{\omega_j}$, which is non-negative since $\delta\geq 0$ and all $\omega_j>0$. This proves that the element of $[d]\choose i$ minimizing the wanted  expression is $[i]$ and the proposition follows.

\end{proof}

Specifying the weight vector to be $\mathbb{1}_{d+1}$ gives an upper bound on the covering minima of the terminal simplex.

\begin{corollary}\label{cor: UB int terminal simplex}
For $d\in \N$, $i\in [d]$, the following inequality holds:
$$\mu_i(T_d)\leqslant \frac{i}{2}\left(1+\frac{d-i}{d+1}\right).$$
\end{corollary}

Let us compare the bounds obtained in Corollaries \ref{cor: UB proj terminal simplex} and \ref{cor: UB int terminal simplex} to the bounds that would be obtained by the following result of Kannan and Lov\' asz.

\begin{lemma}\cite{KannanLovasz}\label{KL lem 2.5}
For a convex body $K\subseteq \R^d$, a lattice $\Lambda \subseteq \R^d$ and $i\in [d]$, the following inequality holds:
$$\mu_{i+1}(K,\Lambda)\leqslant \mu_i(K,\Lambda)+\lambda_{d-i}(K-K, \Lambda).$$
\end{lemma}

One could check that $\lambda_1(T_d-T_d)=\cdots=\lambda_d(T_d-T_d)=\frac{d}{d+1}$, therefore using Lemma \ref{KL lem 2.5} $i-1$ times gives the bound:
$$\mu_i(T_d)\leqslant \frac{1}{2}+\frac{(i-1)d}{d+1}.$$

Since $\frac{d-j}{d-j+1}\leqslant\frac{d}{d+1}$ for all $0\leqslant j \leqslant d$, with equality holding only for $j=0$, we see that the bound from Corollary \ref{cor: UB proj terminal simplex} is stronger than the one obtained from the bound due to Kannan and Lov\'asz.

For small values of $i$ with respect to $d$, the Corollary  \ref{cor: UB proj terminal simplex} will be stronger than Corollary  \ref{cor: UB int terminal simplex}, but as $i$ grows, the latter becomes significantly stronger than the former. 
\section*{Acknowledgements}

The author would like to thank Giulia Codenotti for the introduction to the topic and the many meaningful conversations about the paper. We would also thank Matthias Beck, Ansgar Freyer, Georg Loho and Matthias Schymura for helpful discussions. The autor was partially funded by the Deutsche Forschungsgemeinschaft (DFG, German Research Founda
tion) under Germany’s Excellence Strategy– The Berlin Mathematics Research Center MATH+ (EXC-2046/1, project
ID 390685689, BMS Stipend).

\bibliographystyle{plain}
\bibliography{ref}

\end{document}